\documentclass[11pt]{amsart}
\usepackage[latin1]{inputenc} 
\usepackage[english]{babel}
\usepackage{amsmath}
\usepackage[T1]{fontenc}
\usepackage{amsfonts}  
\usepackage{amssymb}
\usepackage{amsthm}
\usepackage{fancyhdr}
\usepackage{amscd}
\usepackage{latexsym} 
\usepackage{psfrag}
\usepackage{graphicx}
\usepackage[all]{xy}
\usepackage{booktabs}
\usepackage{multicol}
\usepackage{array}
\usepackage{tikz}
\usetikzlibrary{matrix,arrows}
\newcolumntype{C}{>{$}c<{$}}
\newcolumntype{L}{>{$}l<{$}}
\newcolumntype{R}{>{$}r<{$}}

\newcommand{\dual}{\vee}

\newcommand{\bbZ}{\mathbb{Z}}

\newcommand{\bbP}{\mathbb{P}}
\newcommand{\bbC}{\mathbb{C}}
\newcommand{\bbN}{\mathbb{N}}

\newcommand{\bbGr}{\mathbb{G}r}

\newcommand{\sO}{\mathcal{O}}

\newcommand{\sN}{\mathcal{N}}

\newcommand{\isom}{\cong}

\newcommand{\into}{\rightarrow}

\DeclareMathOperator{\rank}{rank}

\DeclareMathOperator{\codim}{codim}

\DeclareMathOperator{\Hilb}{Hilb}

\DeclareMathOperator{\Cat}{Cat}
\DeclareMathOperator{\sym}{Sym}  

\theoremstyle{plain}
\newtheorem{teo}{Theorem}[section]
\newtheorem{teo*}{Theorem}
\newtheorem{propo}[teo]{Proposition}

\newtheorem{cor}[teo]{Corollary}

\newtheorem{lem}[teo]{Lemma}

\theoremstyle{definition} 
\newtheorem{definition}[teo]{Definition}

\theoremstyle{remark}

\newtheorem{remark}[teo]{Remark}

\newtheorem{fn*}{Formule Notevoli}

\begin{document}
\title{Normal Bundle of Rational Curves and Waring Decomposition}
\author{Alessandro Bernardi}
\address{Dipartimento di Matematica, University of Bologna}
\email{alessandro.bernardi8@unibo.it}

% Submission date, in format yyyy/mm/dd (numerical)
\date{\today}
% Abstract
\begin{abstract}
The problem of determining the splitting of the normal bundle of rational space curves has been considered in the 80s in a series of papers by Ghione and Sacchiero and  by Eisenbud and Van de Ven. With our approach we are able to obtain results for curves embedded in $\bbP^m$ for $m\geq 3$ and we find an interesting interplay between the Waring decomposition of binary forms and the splitting of the normal
bundle.

\end{abstract}
\maketitle

% The article itself

%%%%%%%%%%%Introduction%%%%%%%%%%%%%%%%%%%%%%%
%%%%%%%%%%%%%%%%%%%%%%%%%%%%%%%%%%%%%%%%%%%%%%

\section{Introduction}
In this work we address the problem of classifying the rational curves $C\subset \bbP^m$ of degree  $n$ (with $n\geq m$) by the splitting type of their normal bundle $N_{C;\bbP^m}$ and their restricted tangent bundle $T\bbP^m|_C$. 

We choose a projective point of view, by considering the rational normal curve $C_n\subset \bbP^n$ so to view our degree $n$ curves as projections of $C_n$ from a linear space $L\cong \bbP^{k-1}$. The projected curves $C$ lie in a projective space of dimension $m=(n-k)$.

We point out that we are interested to the case with at most ordinary singularities, as in the work of Ghione and Sacchiero (see \cite{Ghione-Sacchiero}).

Since the curve $C$ can be singular, it is necessary to point  out that with $T\mathbb{P}^m \vert C$ and $N_{C,\mathbb{P}^m}$ we mean (with a slight abuse of notation)  $p^*(T\mathbb{P}^m\vert C)$ and $p^*(N_{C,\mathbb{P}^m})$, where $p: \mathbb{P}^1 \rightarrow C$ is given by $p = \pi_L \circ \nu_n$ (the composition of the $n$-th Veronese embedding with the projection from $L$). 

We define the scheme $H^{m,n}$ as the component of the Hilbert scheme $\Hilb_n\bbP^m$ of arithmetic genus zero curves of degree $n$ in $\bbP^m$ containing the smooth curves as an open subset.

We denote with $\sN^n_{n-k}(n_1,...,n_{n-k-1})$ the subscheme of $H^{m,n}$ which {param\-etrizes} the curves such that the splitting type of their normal bundle $N_{C;\bbP^{n-k}}$ is $(n_1,...,n_{n-k-1})$.

We will be working in the Grassmannian $Gr(\mathbb{P}^{k-1},\mathbb{P}^n)$, where we consider its subscheme $N^n_{n-k}(n_1,...,n_{n-k-1})$, which {param\-etrizes} subspaces $L$ such that the curve $C=\pi_L(C_n)$ is such that $N_{C,\mathbb{P}^{n-k}} \simeq \oplus_{i=1}^{n-k-1}\mathcal{O}_{\mathbb{P}^1}(n_i)$. $N^n_{n-k}(n_1,...,n_{n-k-1})$ has the same codimension and number of irreducible\newline components as $\sN^n_{n-k}(n_1,...,n_{n-k-1})$.

We can study directly the basic structure of these subvarieties in the Grassmannian.
 Since the dimension of the Grassmannian is lower than the dimension of the Hilbert scheme, this allows easier computations.

The main advantage of this approach is that we can relate the splitting of the normal bundle of the projected curve to geometric properties of the subspace $L$.

The main novelty is the interplay between the Waring decomposition of forms and some geometric properties of the subspace $L$.

A Waring decomposition of a form $f$ of degree $n$ is an expression $f=\sum^p_{i=1}l^n_i$ where $\deg l_i=1$. Such an expression means that the point corresponding to $f$ in $\bbP^n\isom \bbP(\bbC[s,t]_n)$ belongs to a $(p-1)-$secant space $\bbP^p$ to the rational normal curve. More recent contributions for the simultaneous Waring problem can be found in the work of Iarrobino (see \cite{iarrobinoAncestor}).

If we define for all $n\in \bbN$, $\rho _r^{n,k}\in \mathbb{N}$ as follows:
$$
\rho^{n,k}_r:=\left\lbrace
\begin{array}{cl}
n-3k+r-1 &\mbox{ , for } 3k< n-1 \mbox{ and }  r \leq 2k-1 ;\\
r & \mbox{ , for } 3k\geq n-1 \mbox{ and } 1\leq r<n-k.
\end{array}
\right.
$$

then when $\rho^{n,k}_r\leq (n-k+1)/3$, we can prove (Theorem \ref{mainthmn3}) that the following conditions are equivalent:

\begin{itemize}
\item[i)] the curve $C=\pi_L(C_n)\subset \bbP^{n-k}$ projected from $L\isom \bbP^{k-1}$, has  
$$
N_{\pi_{L}(C_n);\bbP^{n-k}}\isom \sO_{\bbP^1}(n+2)^{\rho^{n,k}_r}\oplus\mathcal{F},
$$ 
where $\deg(\mathcal{F}^{\vee}(n+2))=-2k$ and $\mathcal{F}\isom \bigoplus^{n-1-\rho^{n,k}_r-k}_i \sO_{\bbP^1}(l_i)$, with $l_i\geq n+3$;

\item[ii)] the centre of projection $L\isom\bbP^{k-1}$ is contained in the base locus of a linear system $\Phi$ generated by $\rho_r^{n,k}$ linearly independent $\bbP^{n-3}$'s which intersect $C_n$ in degree $n-2$.

\end{itemize}

%\begin{itemize}
%
%\item[i)] the curve $C=\pi_L(C_n)\subset \bbP^{n-k}$ projected from $L\isom \bbP^{k-1}$, has 
%$$
%N_{\pi_L(C_n);\bbP^{n-k}}\isom \sO_{\bbP^1}(n+2)^{\rho^{n,k}_r}\oplus \sO_{\bbP^1}(n+2+B)^{A-2k+B\cdot A}\oplus\sO_{\bbP^1}(n+3+B)^{2k-B\cdot A}, 
%$$
%
%where $A=n-1-\rho^{n,k}_r-k$ and $B=\lfloor \frac{2k}{A}\rfloor$.
%
%\item[ii)] the centre of projection $L\isom\bbP^{k-1}$ is general among those $\bbP^{k-1}$ contained in the base locus of a linear system $\Phi$ generated by $\rho^{n,k}_r$ linear independent $\bbP^{n-3}$'s which intersect $C_n$ in degree $n-2$.
%\end{itemize}
We want to point out that the summand $\sO_{\bbP^1}(n+2)$ is of particular interest because it has the smallest possible degree in the splitting decomposition of the normal bundle in the ordinary singularities case. Moreover we prove (Theorem \ref{FinallyTheorem}) the following result:

\textbf{Theorem}
\emph{For any $0\leq \alpha \leq n-k-2$, the variety parametrizing rational curves $C$ of degree $n$ in $\bbP^{n-k}$ whose normal bundle $N_{C,\bbP^{n-k}}$ has the summand $\sO_C(n+2)^{\alpha}$ is irreducible.}

This kind of problems has been considered in the 80s in a series of papers  in the case of rational curves in $\bbP^3$ by Ghione and Sacchiero (see \cite{Ghione-Sacchiero} and \cite{Sacchiero2}) and by Eisenbud and Van de Ven (see \cite{Eisenbud-VandeVen1} and \cite{Eisenbud-VandeVen2}).
They used two different points of view but their results can be summarized in the following way: 

The varieties $N^n_{3}(n_1,n_2)$, for any $n_1+n_2=2n-2$, are irreducible of codimension $\codim(N^n_{3}(n_1,n_{2}))=\dim Ext^1(N_{C;\bbP^{3}}N_{C;\bbP^{3}})$, for a general $C\in N^n_{3}(n_1,n_{2})$.  

Few facts in this line of research are known for $\bbP^{m}$, with $m=n-k\geq 4$. The only two papers devoted to this case are the ones of Sacchiero (see \cite{Sacchiero}) and of Miret (see \cite{Miret}).

We work also with the splitting of the restricted tangent bundle $T\bbP^{n-k}|_{\pi_L(C_n)}$. The first results in this direction were given by Ramella in her doctoral thesis (see \cite{Ramella2}) where she proved that the splitting of the normal bundle and of the tangent bundle are surprisingly very little related. Let us mention that Verdier claimed (without proof) in \cite{verdier} a very general result, saying that the varieties of rational curves of degree $n$ in any $\bbP^{n-k}$ with fixed splitting type of their restricted tangent bundle are irreducible of codimension $ext^1(T\bbP^{n-k}|_{C}(-1),T\bbP^{n-k}|_{C}(-1))$.
The proof of this result was clarified later by Ramella (see \cite{Ramella1}).

If we define for all $n\in \bbN$, $\delta^{n,k}_r\in \mathbb{N}$ as follows:
$$
\delta^{n,k}_r:=\left\lbrace
\begin{array}{cl}
r & \mbox{ , for } 2k\leq n \mbox{ and } 1\leq 2r\leq k-1;\\
n-3k+r-1 &\mbox{ , for } 2k > n \mbox{ and }  r \leq n-k-1.
\end{array}
\right.
$$

the when $\delta^{n,k}_r\leq (n-k+1)/2$, we can prove in Theorem \ref{mainthmtg3} a similar result to what is done in Theorem \ref{mainthmn3} for the normal bundle. Namely, we have that the following conditions are equivalent:
\begin{itemize}

\item[i)] the curve $C=\pi_L(C_n)\subset \bbP^{n-k}$ projected from $L\isom \bbP^{k-1}$, has 
$$
T\bbP^{n-k}|_{\pi_L(C_n)}\isom \sO_{\bbP^1}(n+1)^{\delta^{n,k}_r}\oplus\sO_{\bbP^1}(n+1+B)^{A-k+B\cdot A}\oplus \sO_{\bbP^1}(n+2+B)^{k-B\cdot A}), 
$$

where $A=n-\delta^{n,k}_r-k$ and $B=\lfloor \frac{k}{A}\rfloor$.

\item[ii)] the centre of projection $L\isom \bbP^{k-1}$ is contained in the base locus of a linear system $\Phi$ generated by $\delta^{n,k}_r$ linearly independent $\bbP^{n-2}$'s which intersect $C_n$ in degree $n-1$.

\end{itemize}

In the last section we describe the case of rational space curves of degree $5$. In this case the geometric description allows also  to prove the irreducibility and to compute the codimension and the multi-degree of $N^5_3(7,11)$.

In particular we prove in Theorem \ref{thmC5} (p.\pageref{thmC5}) the following result:

$N^5_2(7,11)$ is an irreducible variety of codimension $3$ and bidegree $(1,6)$ formed by the lines $L\isom\bbP^{1}$ that belong to a $3-$secant plane to the rational normal curve $C_5$.

More recent contributions in the case of plane curves can be found in the work of Gimigliano, Harbourne and Ida (see \cite{GHI}) where they related the above problems to the problem of determining the resolution of plane fat point schemes.

%%%%%%%%%%%%Apolarity and Waring's Problem%%%%%%%%%%%%%%%%%%
%%%%%%%%%%%%%%%%%%%%%%%%%%%%%%%%%%%%%%%%%%%%%%%%%%%%%%%%%%%%

\section{Apolarity and Waring's Problem}
\subsection{Catalecticant and Apolarity Setup}
In the following section we want to fix the notation and to recall some results about connections with apolarity theory, catalecticant homomorphisms for binary forms and secant varieties of rational normal curves (see \cite{iarrobinokanev}).
\subsubsection{Contraction Action and Catalecticant Morphism}
Let $V$ be a complex vector space with $\dim_{\bbC}V=2$ and basis $x_0,x_1$. We consider the ring of homogeneous polynomial $S=\bigoplus_{i\geq 0}\sym^iV$, where $\sym^iV$ is the $i-$th symmetric power of $V$. The dual basis of $V^{\vee}$ can be denoted by $\partial_0=\frac{\partial}{\partial x_0},\partial_1=\frac{\partial}{\partial x_1}$ and $T=\bigoplus_{i\geq 0}\sym^iV^{\vee}$ is the dual ring of $S$. The $\partial _i$'s are operators acting on the $x_j$'s and viceversa. This is called the contraction action.

\begin{definition}
If we fix a form $f\in \sym^nV$, we can associate to $f$ a catalecticant homomorphism for all $1\leq e\leq n-1$:
$$
C_f(e,n-e):\sym^{n-e}V^{\dual}\into \sym^{e}V,\;\; \mbox{ where } \phi\into \phi\circ f,
$$
the corresponding matrix is called the catalecticant matrix $\Cat_f(e,n-e)$, where we choose as basis for $S_e$ the monomials $x^{[E]}=\frac{1}{e_0!e_1!}x_0^{e_0}x_1^{e_1}$ with $E=(e_0,e_1)\in \bbN^{2}$ and $e_0+e_1=e$. For a binary form $f\in \sym^nV$ :
$$
f=a_0x_0^n+...+\binom{n}{d}a_dx^{n-d}_0x^{d}_1+...+a_nx^n_1,
$$
the corresponding $e-$th catalecticant matrix is : 
$$
\Cat_f(e,n-e)=
\left(
\begin{array}{cccccc}
a_0&a_1&\cdots&a_{n-e-1}&a_{n-e}\\
a_1&a_2&\cdots&a_{n-e}&a_{n-e+1}\\
\vdots&\vdots&\ddots&\vdots&\vdots\\
a_{e-1}& a_e &\cdots&a_{n-2}&a_{n-1}\\
a_e& a_{e+1} &\cdots&a_{n-1}&a_n
\end{array}\right).
$$
\end{definition}
\begin{remark}
Clearly the transpose $\Cat_f(e,n-e)^t$ satisfies:
$$
\Cat_f(e,n-e)^t=\Cat_f(n-e,e).
$$
\end{remark}

\begin{definition}
To each element $f\in S_n$ we can associate the ideal $I_f=Ann(f)$ in $T$, consisting of polynomials $\phi$ such that $\phi\circ f=0$. We call $\phi$ and $f$ \emph{apolar} to each other and $I_f$ is the \emph{apolar ideal} of $f$.
\end{definition}

One associates to $f$ also the quotient algebra $A_f=T/I_f$. Macaulay called the ideal $I_f$ a "principal system", while we know them as Gorenstein ideals, since $A_f$ is a Gorenstein Artin algebra. 

If $p=(a_0,a_1)\in\bbC^{2}$, let $L_p$ denote the linear form:
$$
L_p=a_0x_0+a_1x_1\in S_1.
$$
\begin{remark}
We have the following useful equality:
$$
\phi\circ L^{n}_p=\phi(p)L^{n-e}_p,
$$
for all $\phi\in T_e$, $e\leq n$ and any $L_p$. With abuse of notation we will write $\phi(L_p)$ for $\phi(p)$.
\end{remark}

\begin{definition}
Consider the forms $f\in S_n$ that can be written as a sum:
$$
f=L^{n}_1+...+L^{n}_s,
$$
for some choice of the linear forms $L_1,...,L_s\in S_1$. They form the image of the regular map:
$$
\mu:\overbrace{S_1\times...\times S_1}^{s-times}\into S_n,
$$
defined by $\mu(L_1,...,L_s)=L^{n}_1+...L^{n}_s$. Let us denote by $PS(s,n)$ this image. Its algebraic closure is an irreducible affine variety and it is invariant under multiplication by elements of $\bbC^*$.
\end{definition}
\begin{remark}\label{secantremark}
If we consider the n-th Veronese map $\nu_n:\bbP^1\into \bbP^n$, its image $C_n=\nu_n(\bbP^{1}),\;\; \bbP^1\isom\bbP(S_1)$ is the rational normal curve of degree $n$, $C_n\subset \bbP^n\isom \bbP(S_n)$. If $s\leq \dim_{\bbC}T_n=n+1$, then for any forms $L_1,...,L_s$ the projectivization of the span $<L^{n}_1,...,L^{n}_s>$ is an $(s-1)-$plane that intersects the rational normal curve in the points $\nu_n(<L_i>)=<L^{n}_i>$, $i=1,..,s$. Thus $\overline{\bbP PS(s,n)}$ is exactly the $s-$secant variety to the rational normal curve: $\sigma_{s}(\nu_n(\bbP^1))$.
\end{remark}
\begin{remark}
If $f\in PS(s,n)$ and $f=L^{n}_1+...+L^{n}_s$, then for $e\leq n$ and every $\phi\in T_{n-e}$, we have:
$$
\phi\circ f=\phi(L_1)L_1^{e}+...+\phi(L_s)L^{e}_s.
$$
This shows that $S_e(f)$ the image of the catalecticant homomorphism $C_f(e,n-e)$ has dimension $\leq s$.  Hence the $(s+1)\times(s+1)$ minors of the catalecticant matrices $\Cat_f(e,n-e)$ vanish on $PS(s,n)$.
\end{remark}

\subsection{The Grassmannians of secant varieties of curves}
It is well known that curves $C$ in $\bbP^n$ are not defective (see for example \cite{Zak}), i.e the secant varieties $\sigma_r(C)$ all have the expected dimension $\min\{n,2r+1\}$.
Moreover, we have the following well known result (see for example \cite{harrisalggeo}):
\begin{propo}
A smooth, non degenerate projective curve $C\subset \bbP^n$ is projected isomorphically from a point $p\in \bbP^n$ to $\bbP^{n-1}$ if and only if $p$ does not belong to the secant variety $\sigma_2(C)$ of $C$.
\end{propo}
We are interested in considering projections of curves from some linear subspace. In this case if a linear span $H$ of $r+1$ points of $C$ contains the center of projection, then it determines a $(r+1)-$secant space for the image $C'$ of dimension less then $r$. In analogy with the theory of secant varieties $\sigma_{r}(C)$ one may ask about the expected dimension of these Grassmannians of secant spaces.

\begin{definition}
Let $C\subset \bbP^n$ be an irreducible non degenerate curve. We define the secant varieties $G_r(C)\subset Gr(\bbP^r,\bbP^n)$, where:
$$
G_r(C)=\overline{\{H\subset \bbP^n:H\mbox{ is the span of } r+1 \mbox{ independent points of } C\}},
$$
and
$$
\sigma_r^{n,k}(C) =\overline{\{p\in \bbP^n: p\in H \mbox{ for some } H\in G_r(C)\}},
$$        
Since $C$ is irreducible, then $G_r(C)$ is irreducible of dimension $r+1$.
\end{definition}
If one considers, in the incidence variety of $Gr(\bbP^{r},\bbP^n)\times \bbP^n$, the subsets:
$$
I(C)=\{(H,p):p\in H, H \mbox{ is spanned by } r+1 \mbox{ independent points of } C\},
$$
then $G_r(C),\sigma_r^{n,k}(C)$ correspond to the closures of the images of the two natural projections $I(C)\into Gr(\bbP^{r},\bbP^n)$ and $I(C)\into \bbP^n$. In particular $\dim I(C)=2r+1$ and since $C$ is not defective, the map $I(C)\into \bbP^n$ is generically finite when $2r+1\leq n$, while otherwise it has general fibers of dimension $2r+1-n$.
\begin{definition}
We denote by $G_{s,r}(C)$ the following subset of $Gr(\bbP^s,\bbP^n)$:
$$
G_{s,r}(C)=\overline{\{ h\in Gr(\bbP^s,\bbP^n): h\subset H \mbox{ for some } H\in G_r(C) \}}.
$$
These objects are the Grassmannians of secant varieties of $C$. Observe that $G_{0,r}(C)$ coincides with $\sigma_r(C)$.
\end{definition}
The elements of $G_{s,r}(C)$ are contained in the Grassmannian of $s-$spaces of some $k-$spaces $H\in G_r(C)$. Thus we always have:
$$
\dim G_{s,r}(C)\leq \dim Gr(\bbP^s,\bbP^r)+G_r(C)=(s+1)(r-s)+k+1.
$$  
Furthermore:
$$
\dim G_{s,r}(C)\leq \dim Gr(\bbP^s,\bbP^n)=(s+1)(n-s).
$$
\begin{definition}
We define the expected dimension of $G_{s,r}(C)$ as:
$$
expdim(G_{s,r}(C)=\min \{ (s+1)(n-s), (s+1)(r-s)+r+1 \}.
$$
\end{definition}
Chiantini and Ciliberto proved that the actual dimension of $G_{s,r}(C)$ is always equal to the expected one.
\begin{teo}[\cite{chiantiniciliberto}]\label{chiantiniciliberto}
The dimension of $G_{s,r}(C)$ is equal to the expected dimension:
$$
\min \{ (s+1)(n-s), (s+1)(r-s)+r+1 \}.
$$
\end{teo}                                
There is a relationship between $G_{s,r}(C_n)$ and simultaneous additive decompositions of a set of binary forms, as this easy result shows:
\begin{propo}\label{dec.simultanee}
Let $f_0,...,f_s\in S_n$. They have a simultaneous additive decomposition, i.e. there exist $L_1,...,L_{r+1}\in S_1$ such that:
$$
f_i=c_1^iL^n_1+...+c^i_{r+1}L^n_{r+1},\;\; i=0,...,s,
$$
with $c^i_j\in \bbC$, if and only if the corresponding points $p_{f_i}\in\bbP^n$ belong to some $(r+1)-$secant space $H$ to $C_n$. If they are linearly independent, we have $<p_{f_0},...,p_{f_s}>\in G_{s,r}(C_n)$. 
\end{propo}

%\newpage
%%%%%%%Rational Curves of degree n%%%%%%%%%%%%%%%%%%
%%%%%%%%%%%%%%%%%%%%%%%%%%%%%%%%%%%%%%%%%%%%%%%%%%%%

\section{Rational Curves of degree $n$ in $\bbP^{n-k}$}

Let $C\subset \bbP^{n-k}$, $n-k\geq 2$, be a rational curve of degree $n$ with only ordinary singularities. By Grothendieck-Segre's theorem (see \cite{Groth1}) we can write $N_{C;\bbP^{n-k}}=\bigoplus^{n-k-1}_{i=1}\sO_{\bbP^1}(n+d_i)$. On the other hand we have $c_1(N_{C;\bbP^{n-k}})=(n-k+1)n-2$, so $\sum^{n-k-1}_{i=1}d_i=2n-2$.
Let $\psi:\bbP^1\into \bbP^{n-k}$ be the morphism defined by $\psi_i\in \Gamma(\bbP^1,\sO_{\bbP^1}(n))$. Let $C=\psi(\bbP^1)$, we suppose that it has only ordinary singularities, i.e. the map of differential fibre bundles $\Omega\psi:\psi^*\Omega_{\bbP^{n-k}}\into \Omega_{\bbP^1}$ is surjective.

\begin{definition}
We denote with  $N^n_m(n_1,...,n_{m-1})\subset Gr(\mathbb{P}^{k-1},\mathbb{P}^n)$ the subscheme {param\-etrizing} linear spaces $L\subset \mathbb{P}^n$ such that $\pi_L(C_n)$ has $(n_1,...,n_{m-1})$ has splitting type of their normal bundle; we denote with $T^n_m(t_1,...,t_{m})$ the analogous subscheme for the splitting type of the restricted tangent bundle. 

\end{definition}

%Note that $SL(m+1)$ acts on both $H^{m,n}$ and $V^{m,n}$, and that the subschemes $N^n_{m}(n_1,...,n_{m-1})$ are invariant under this action. Furthermore the general isotropy subgroup of this action on $M^{m,n}$ is $SL(2)$. The action of $SL(m+1)$ is free on $V^{m,n}$ and an open subscheme of the quotient $V^{m,n}//SL(m+1)$.
%An open subscheme of the quotient of $V^{m,n}$ for $SL(m+1)$ has dimension $(n+1)(m+1)-1-((m+1)^2-1)=(m+1)(n-m)$. Indeed this quotient is isomorphic to an open subset  of the Grassmannian $Gr(\bbP^{k-1},\bbP^n)$ for $m=n-k$.
%
%So in this paper we choose to work directly on the Grassmannian $Gr(\bbP^{k-1}, \bbP^n)$.
%
%One advantage of working directly on the Grassmannian $Gr(\bbP^{k-1}, \bbP^n)$ is that the irreducible components and the codimension of the varieties $N^n_{m}(n_1,...,n_{m-1})/SL(m+1)$ (parametrizing subspaces $L$ such that the curve obtained by projecting from $L$
%has normal bundle isomorphic to $\bigoplus^{m-1}_{i=1}\sO_{\bbP^1}(n_i)$ )  remain the same of $N^n_{m}(n_1,...,n_{m-1})$ and similarly for $T^n_{m}(t_1,...,t_{m})$. Then we can study directly the basic structures of these subvarieties in the Grassmannian. Since now the dimension of the Grassmannian is lower than the dimension of the Hilbert scheme, this allows easier computations.
%As we noticed in the introduction, we will work directly in the Grassmannian $Gr(\bbP^{k-1}, \bbP^n)$.
%The main advantage of this approach is that we can relate the splitting of the normal bundle of the projected curve to geometric properties of the centre of projection $L$.

Let $C_n\subset \bbP^n$ be the rational normal curve of degree $n$ and $\nu_n:\bbP^1\into \bbP^n$ be the Veronese map, so $C_n=\nu_n(\bbP^1)$. Let $\pi_L(C_n)$ be the rational curve obtained from $C_n$ by projection from a $(k-1)$-dimensional linear subspace $L\subset \bbP^n$ on $\bbP^{n-k}\subset \bbP^n$. We will suppose that $\pi_L(C_n)$ has only ordinary singularities. 
 Let  $J(\nu_n)$ be the Jacobian matrix of $\nu_n$:
\[
J(\nu_n)=
\left(
\begin{array}{ccccc}
ns^{n-1} & \dots & t^{n-1} & 0 \\
0 & s^{n-1} & \dots & nt^{n-1}
\end{array}
\right).
\]
Since we are mainly interested in studying the splitting of the normal bundle of rational curves, we restrict our attention to the case $k<n-2$.

The Euler's exact sequence on $\bbP^{n-k}$, restricted to $\pi_L(C_n)$ (see \cite{Hartshorne}, \cite{Okonek}) and the usual one for the normal bundle on $\pi_L(C_n)$, give rise to the following diagram, where $\deg \sO_{\pi_L(C_n)}(1)=1$ :

\begin{tikzpicture}
\matrix(m)[matrix of math nodes,
row sep=2em, column sep=1em,
text height=1.5ex, text depth=0.25ex]
{&0&[1.5cm]0& &\\
& \sO_{\pi_L(C_n)}  & \sO_{\pi_L(C_n)} &  &\\
 & \sO_{\pi_L(C_n)}(1)^2  & \sO_{\pi_L(C_n)}(n)^{n-k+1}& &\\
0  & \sO_{\pi_L(C_n)}(2) & T\bbP^{n-k}|_{\pi_L(C_n)}& {N'}_{\pi_L(C_n) ; \bbP^{n-k}}  & 0\\
 & 0 & 0 & & \\
};
\path[->,font=\scriptsize,>=angle 90]
(m-1-2) edge node[auto]  {} (m-2-2)
(m-1-3) edge node[auto]  {} (m-2-3)
(m-2-2) edge node[auto]  {=}(m-2-3)
        edge node[auto]  {} (m-3-2)
(m-2-3) edge node[auto]  {} (m-3-3) 
(m-3-2) edge node[auto]  {$J(\pi_L\circ \nu_n)$} (m-3-3)
        edge node[auto]   {} (m-4-2)
(m-3-3) edge node[auto]   {} (m-4-3)   
(m-4-3) edge node[auto]   {} (m-5-3)     
(m-4-1) edge node[auto] {} (m-4-2)
(m-4-2) edge node[auto]  {} (m-4-3)
(m-4-2) edge node[auto] {} (m-5-2)
(m-4-3) edge node[auto] {} (m-4-4)
(m-4-4) edge node[auto] {} (m-4-5);        
\end{tikzpicture}

where ${N'}_{\pi_L(C_n) ; \bbP^{n-k}}$ is the equisingular normal sheaf (see Def.3.4.5 \cite{sernesi}).
Therefore the following exact sequence holds:

\begin{equation}\label{succ:12}
\begin{tikzpicture}
\matrix(m)[matrix of math nodes,
row sep=0em, column sep=1em,
text height=1.5ex, text depth=0.25ex,baseline=(current bounding box.center)]
{0& \sO_{\pi_L(C_n)}(1)^2  &[1.5cm] \sO_{\pi_L(C_n)}(n)^{n-k+1} & {N'}_{\pi_L(C_n) ;\bbP^{n-k}}  & 0\\
};
\path[->,font=\scriptsize,>=angle 90]    
(m-1-1) edge node[auto] {} (m-1-2)
(m-1-2) edge node[auto]  {$J(\pi_L\circ \nu_n)$} (m-1-3)
(m-1-3) edge node[auto] {} (m-1-4)
(m-1-4) edge node[auto] {} (m-1-5);        
\end{tikzpicture}
\end{equation}

\begin{remark}
We can observe that if $\pi_L(C_n)$ has only ordinary singularities, then the map of differentials is surjective (see \cite{Ghione-Sacchiero}), so ${N'}^{\vee}_{\pi_L(C_n);\bbP^{n-k}}$ is a vector bundle. Since we consider only cases with ordinary singularities, we will indicate with$(N_{\pi_L(C_n);\mathbb{P}^{n-k}})$ the bundle  $(\pi_L\circ \nu_n)^*(N'_{\pi_L(C_n);\mathbb{P}^{n-k}})$, even if it is a bundle over $\mathbb{P}^1$.

\end{remark}
\begin{remark}
We observe that $deg  N_{\pi_{L}(C_n);\bbP^{n-k}}(-n)=2n-2$, equivalently $deg  N_{\pi_{L}(C_n);\bbP^{n-k}}=n^2-(k-1)n-2$. It is a vector bundle of $\rank n-k-1$. 

Moreover by Grothendieck-Segre's theorem (see \cite{Groth1}) $N_{\pi_{L}(C_n);\bbP^{n-k}}$ splits as:
$$
\bigoplus^{n-k-1}_{i=1}\mathcal{O}_{\mathbb{P^1}}(n_i) \mbox{ with } n_i\in \bbZ, \mbox{ such that } \sum^{n-k-1}_{i=1}n_i=n^2-(k-1)n-2,
$$ 
where, without loss of generality, we can take $n_1\leq \dots \leq n_{n-k-1}$.
\end{remark}
\begin{remark}
It's easy to show that:
\[
N_{C_n;\bbP^{n}}=\mathcal{O}_{\mathbb{P^1}}(n+2). 
\]

\end{remark}
Let us denote by:
$$
Syz(J(\nu_n))=
\left(
\begin{array}{ccccccccccc}
t^2 & -2st & s^2 & 0 &0& 0&\dots & \dots & 0\\
0 & t^2 & -2st & s^2 & 0& 0& \dots & \dots & 0 \\
0 & 0 & t^2 & -2st & s^2& 0&  \dots & \dots &0 \\
\vdots & \ddots & \ddots & \ddots &\ddots & \ddots &\ddots & \ddots & \vdots \\
0 & \dots & \dots & \dots &\dots & 0 &t^2& -2st & s^2
\end{array}
\right),
$$

the syzygy of the Jacobian matrix $J(\nu_n)$, where $Syz(J(\nu_n))$ is a $(n-1)\times (n+1)-$matrix.

Moreover we have an exact sequence obtained by the projection $\pi_L$ twisted by $\sO_{\pi_L(C_n)}(-n)$:
\begin{equation}
\xymatrix{
0 \ar[r] & \sO_{\pi_{L}(C_n)}^{k} \ar[r]^P & \sO_{\pi_{L}(C_n)}^{n+1} \ar[r] & \sO_{\pi_{L}(C_n)}^{n-k+1} \ar[r] & 0,}
\end{equation}
where:
$$
P=
\left[
\begin{array}{ccccccccc}
p_1&\cdots&p_k
\end{array}
\right],
$$
where $p_i=(a^i_0,...,a^i_n)\in \bbP^n$, for $i=1,\dots, k$, and $p_1,...,p_k$ is a set of points which generate the subspace $L\subset \bbP^n$ .

Let $V$ be a 2-dimensional complex vector space, and consider the following diagram:
\begin{equation}
\xymatrix{
&&0\ar[d]&\\
& & \sO^{k}_{\pi_{L}(C_n)} \ar[d]^P &   &  \\ 
0 \ar[r] & \sO_{\pi_{L}(C_n)}(-n+1)^2 \ar[r]^{J(\nu_n)} & \sO_{\pi_{L}(C_n)}^{n+1} \ar[d] \ar[r]^{Syz(J(\nu_n))} & \sO_{\pi_{L}(C_n)}(2)^{n-1}  \ar[r] & 0 \\
0 \ar[r] & \sO_{\pi_{L}(C_n)}(-n+1)^2 \ar[r]^{J(\pi_L\circ\nu_n)} & \sO_{\pi_L(C_n)}^{n-k+1} \ar[d] \ar[r] & N_{\pi_L(C_n);\bbP^{n-k}}(-n) \ar[r] & 0\\
&&0&,
}
\end{equation}
via $(\pi_L\circ \nu_n)^*)$, we can complete and write down (3) in a more invariant way on $\mathbb{P}^1$:

\begin{equation}\label{diagr:normCODIMK}
\begin{tikzpicture}
\matrix(m)[matrix of math nodes,
row sep=1.5em, column sep=1em,
text height=1.5ex, text depth=0.25ex]
{&&[1.75em]0&[1.75em]0\\
& & \bbC^{k}\otimes\sO_{\mathbb{P}^1}& \bbC^{k}\otimes\sO_{\mathbb{P}^1}  &  \\ 
0 & V\otimes\sO_{\mathbb{P}^1}(-n+1)  & \sym^nV\otimes\sO_{\mathbb{P}^1}  & \sym^{n-2}V\otimes \sO_{\mathbb{P}^1}(2) & 0 \\
0  & V\otimes\sO_{\mathbb{P}^1}(-n+1)  & \frac{\sym^nV}{\bbC^{k}}\otimes\sO_{\mathbb{P}^1}  & N_{\pi_L(C_n);\bbP^{n-k}}(-n)  & 0\\
&&0&0\\};
\path[->,font=\scriptsize,>=angle 90]
(m-1-3) edge node[auto]  {} (m-2-3)
(m-1-4) edge node[auto]  {} (m-2-4)
(m-2-3) edge node[auto]  {$\isom$}(m-2-4)
        edge node[auto]  {P} (m-3-3)
(m-2-4) edge node[auto]  {$(\mathcal{N}^L_{n,k})^t$} (m-3-4) 
(m-3-1) edge node[auto]  {} (m-3-2)
(m-3-2) edge node[auto]  {$J(\nu_n)$} (m-3-3)
(m-3-3) edge node[auto]  {$Syz(J(\nu_n))$} (m-3-4)
(m-3-4) edge node[auto]  {} (m-3-5)
(m-4-1) edge node[auto] {} (m-4-2)
(m-4-2) edge node[auto]  {$J(\pi_{L}\circ\nu_n)$} (m-4-3)
(m-4-3) edge node[auto] {} (m-4-4)
(m-4-4) edge node[auto] {} (m-4-5)
(m-3-3) edge node[auto] {} (m-4-3)
(m-3-4) edge node[auto] {} (m-4-4)
(m-4-3) edge node[auto] {} (m-5-3)
(m-4-4) edge node[auto] {} (m-5-4);        
\end{tikzpicture}
\end{equation}

where the map $(\mathcal{N}^L_{n,k})^t$ is:
$$
(\mathcal{N}^L_{n,k})^t=
Syz(J(\nu_n))\cdot \left(
\begin{array}{ccccc}
a^1_0 & \dots & a^{k}_0 \\
\vdots & \ddots & \vdots \\
a^1_n & \dots & a^{k}_n
\end{array}
\right)=
$$
$$
\left(
\begin{array}{ccccccc}
a^1_0 t^2-2a^1_1 ts+a^1_2 s^2 & a^1_1 t^2-2a^1_2 ts+a^1_3 s^2 & \dots & a^1_{n-2} t^2-2a^1_{n-1} ts+a^1_n s^2\\
\vdots & \vdots & \ddots & \vdots \\
a^{k}_0 t^2-2a^{k}_1 ts+a^{k}_2 s^2 & a^{k}_1 t^2-2a^{k}_2 ts+a^{k}_3 s^2 & \dots & a^{k}_{n-2} t^2-2a^{k}_{n-1} ts+a^{k}_n s^2
\end{array}
\right)^t.
$$

The last exact column of (\ref{diagr:normCODIMK})
gives us some information about the splitting type of ${N}_{\pi_L(C_n),\bbP^{n-k}}$:

\begin{lem}\label{boundaryconditionsplittingnormalbundlecodimk}
If $\pi_L(C_n)$ has only ordinary singularities, then the splitting type of ${N'}_{\pi_L(C_n),\bbP^{n-k}}$ must be $(n_1,...,n_{n-k-1})$ with $n+2\leq n_1\leq ...\leq n_{n-k-1} \leq n+2+2k$. 
\end{lem}

If we consider the exact cohomology sequence which is the dual of the last exact column of (\ref{diagr:normCODIMK}), we get from a more invariant point of view:
\[
H^0(N^{\vee}_{\pi_L(C_n);\bbP^{n-k}}(n+2))\isom \ker\{\sym^{n-2}V^* \xrightarrow{N^L_{n,k}} \bbC^{k}\otimes \sym^{2}V  \}
\] 
while $H^1(N^{\vee}_{\pi_L(C_n);\bbP^{n-k}}(2n-2))$ is the cokernel of $N^L_{n,k}$ which is the $3k\times (n-1)$ matrix:
$$
N^L_{n,k}=
\left(
\begin{array}{ccccc}
a^1_0 & \dots & a^1_{n-2}\\
-2a^1_1 & \dots & -2a^1_{n-1}\\
a^1_2 & \dots & a^1_{n}\\
\vdots & \ddots & \vdots \\
a^{k}_0 & \dots & a^{k}_{n-2}\\
-2a^{k}_1 & \dots & -2a^{k}_{n-1}\\
a^{k}_2 & \dots & a^{k}_n
\end{array}
\right).
$$
\begin{remark}\label{boundaryconditionsplittingnormalbundle2codimk}
\begin{itemize}
\item[i.] We have that $\deg(N^{\vee}_{\pi_L(C_n);\bbP^{n-k}}(n+2))= -2k$ and 
$$
h^0(N^{\vee}_{\pi_L(C_n);\bbP^{n-k}}(n+2))= n-1 - \rank(N^L_{n,k})=\dim \ker (N^L_{n,k}).
$$ 
Therefore we have:  
$$
2\leq \rank(N^L_{n,k})\leq \min\{n-1,3k\},
$$ 
so 
$$
0\leq h^0(N^{\vee}_{\pi_L(C_n);\bbP^{n-k}}(n+2))\leq n-3,
$$
but as $\rank(N^{\vee}_{\pi_L(C_n);\bbP^{n-k}}(n+2))=n-k-1$ we have that, by Grothendieck-Segre's theorem (see \cite{Groth1}): 
$$
N^{\vee}_{\pi_{L}(C_n);\bbP^3}(n+2) \mbox{ splits in } \bigoplus^{n-k-1}_{i=1}\sO_{\bbP^1}({n'}_i),
$$ 
with 
$$
-2k\leq {n'}_1\leq ... \leq {n'}_{n-k-1}\leq 0 \mbox{ and } {n'}_1+...+{n'}_{n-k-1}=-2k.
$$

\item[ii.] We have that:
$$
\rank(N^L_{n,k})=\min\{n-1,3k\} \mbox{ if and only if } h^0(N^{\dual}_{\pi_{L}(C_n);\bbP^{n-k}}(n+2))=0,
$$
which is equivalent to 
$$
N^{\dual}_{\pi_{L}(C_n);\bbP^{n-k}}(n+2)\isom \sO_{\bbP^1}({n'}_0)\oplus ... \oplus \sO_{\bbP^1}({n'}_{n-k-2}),  
$$
with all  ${n'}_i\neq 0$.

Moreover:
$$
\rank(N^L_{n,k})< \min\{n-1,3k\}
$$
$$
 \mbox{ if and only if } 
$$
$$ 
h^0(N^{\dual}_{\pi_{L}(C_n);\bbP^{n-k}}(n+2))=n-1-\rank(N^L_{n,k}), 
$$
which, when $k\leq r$, is equivalent to: 
$$
N^{\dual}_{\pi_{L}(C_n);\bbP^{n-k}}(n+2)\isom\sO^{n-1-\rank(N^L_{n,k})}_{\bbP^1}\oplus \mathcal{F}^{\vee}(n+2),
$$ 
where $\rank(\mathcal{F})=\rank(N^L_{n,k})-k$  and $\deg(\mathcal{F}^{\vee}(n+2))=-2k$ or, equivalently, $\deg(\mathcal{F})=(r-k)(n+2)+2k$.

But we have also:
$$
\rank(N^L_{n,k})\geq k+1,
$$
since, otherwise, some ${n'}_i$ should be $\geq 1$, so that some $n_i$ should be $\leq n+1$, but this is impossible by Lemma \ref{boundaryconditionsplittingnormalbundlecodimk}.  
\end{itemize}
\end{remark}

By the above considerations it follows that:

\begin{propo}\label{prop:splittingnormal}
If $\pi_L(C_n)$ has only ordinary singularities, then 
$$
N_{\pi_{L}(C_n);\bbP^{n-k}}\isom \sO_{\bbP^1}(n+2)^{n-1-\rank(N^L_{n,k})}\oplus\mathcal{F},
$$ 
where $\deg(\mathcal{F}^{\vee}(n+2))=-2k$ and $\mathcal{F}\isom \bigoplus^{\rank(N^L_{n,k})-k}_i \sO_{\bbP^1}(l_i)$, with $l_i\geq n+3$.
\end{propo}

As in the case of normal bundle we obtain an exact sequence for restricted tangent bundle; let $T\mathbb{P}^{n-k}\vert_{\pi_L(C_n)}$ denote $(\pi_L\circ \nu_n)^*(T\mathbb{P}^{n-k}\vert _{\pi_L(C_n)})$, then:
\begin{equation}\label{succ:T}
\begin{tikzpicture}
\matrix(m)[matrix of math nodes,
row sep=0em, column sep=1em,
text height=1.5ex, text depth=0.25ex,baseline=(current bounding box.center)]
{0  & (T\bbP^{n-k}|_{\pi_L(C_n)})^{\vee}(n+1)  & \sym^{n-1}V\otimes\sO_{\bbP^1}^n & \bbC^k\otimes\sO_{\bbP^1}(1) & 0,\\
};
\path[->,font=\scriptsize,>=angle 90]    
(m-1-1) edge node[auto] {} (m-1-2)
(m-1-2) edge node[auto]  {} (m-1-3)
(m-1-3) edge node[auto] {$\mathcal{T}^L_{n,k}$} (m-1-4)
(m-1-4) edge node[auto] {} (m-1-5);        
\end{tikzpicture}
\end{equation}

where we have indicated with $T^L_{n,k}$ the $2k\times n$ matrix:
$$
T^L_{n,k}=
\left(
\begin{array}{ccccc}
a^1_0 & \dots & a^1_{n-1}\\
-a^1_1 & \dots & -a^1_{n}\\
\vdots & \ddots & \vdots \\
a^{k}_0 & \dots & a^{k}_{n-1}\\
 -a^{k}_1 & \dots & -a^{k}_n
\end{array}
\right).
$$

So we have that $\deg((T\bbP^{n-k}|_{\pi_L(C_n)})^{\vee}(n+1))=-k$ and 
$$
h^0((T\bbP^{n-k}|_{\pi_L(C_n)})^{\vee}(n+1))= n - \rank(T^L_{n,k})=\dim \ker (T^L_{n,k}).
$$ 
As $\rank((T\bbP^{n-k}|_{\pi_L(C_n)})^{\vee}(n+1))=n-k$, we have that $(T\bbP^{n-k}|_{\pi_L(C_n)})^{\vee}(n+1)$ splits in to $\sO_{\bbP^1}({t'}_1)\oplus ... \oplus \sO_{\bbP^1}({t'}_{n-k})$, by Grothendieck-Segre's theorem (see \cite{Groth1}) with ${t'}_1\leq ...\leq {t'}_{n-k}\leq0$ and ${t'}_1+..+{t'}_{n-k}=-k$. 

Therefore we have that:
$$
\rank T^L_{n,k}\geq k+1,
$$
otherwise some ${t'}_i$ must be $\geq 1$ which is impossible.

We recall that:
\[
T\bbP^{n}|_{C_n}\simeq \sO^{n}_{\bbP^1}(n+1).
\]

By the considerations above it follows that:
\begin{propo}\label{splittingrank}
$T\bbP^{n-k}|_{\pi_L(C_n)}\isom \sO_{\bbP^1}(n+1)^{n-\rank(T^L_{n,k})}\oplus\mathcal{F}$, where $\deg(\mathcal{F}^{\vee}(n+1))=-k$ and $\mathcal{F}\isom \bigoplus^{\rank(T^L_{n,k})-k}_{i=1}\sO_{\pi_L(C_n)}(l_i)$, with $l_i\geq n+2$.
\end{propo}

\begin{remark}
The rank of $N^L_{n,k}$ (resp. $T^L_{n,k}$) does not depends on $L$ and not on the choice of the points.
Let $F_{p_i}$ be the binary form of degree $n$ which corresponds to the point $p_i\in\bbP^n$. If we indicate with $Cat_{F_{p_i}}(2,n-2)$ or $Cat_{p_i}(2,n-2)$ the Hankel $3\times(n-1)$ matrix of $F_{p_i}$ we have:
$$
\rank N^L_{n,k}=
\rank \left(
\begin{array}{c}
Cat_{F_{p_1}}(2,n-2) \\ 
\vdots \\
Cat_{F_{p_{k}}}(2,n-2)
\end{array}
\right),
$$
resp.
$$
\rank T^L_{n,k}=
\rank \left(
\begin{array}{c}
Cat_{F_{p_1}}(1,n-1) \\ 
\vdots \\
Cat_{F_{p_{k}}}(1,n-1)
\end{array}
\right),
$$
\end{remark}

\begin{remark}
It's clear from the above consideration that the same value of $\rank(N^L_{n,k})$ (resp. $T^L_{n,k}$) may correspond to different splitting types of ${N}_{\pi_L(C_n);\bbP^{n-k}}$ (resp. $T\bbP^{n-k}|_{\pi_L(C_n)}$).
\end{remark}

\section{Main Theorems}

When considering the matrix $N^L_{n,k}$, we can notice that there are two different cases:
\begin{itemize}
\item[i)] $3k\geq n-1$, so $\frac{n-1}{3}\leq k<n-3$.
\item[ii)] $3k<n-1$, so $k<\frac{n-1}{3}$.
\end{itemize}
\begin{definition}
For all $n\in \bbN$, let us define $\rho^{n,k}_r\in\bbN$ as follows:
$$
\rho^{n,k}_r:=\left\lbrace
\begin{array}{cl}
r & \mbox{ , for } 3k\geq n-1 \mbox{ and } 1\leq r<n-k;\\
n-3k+r-1 &\mbox{ , for } 3k< n-1 \mbox{ and }  r \leq 2k-1.
\end{array}
\right.
$$

\end{definition}

\begin{remark}
In the above definition, the conditions on $r$ come from 
$$
\rank(N^L_{n,k})\geq k+1.
$$
\end{remark}

\begin{lem}
Let $\rho^{n,k}_r\leq (n-k+1)/3$. 
%The centre of projection $L\isom\bbP^{k-1}$ belongs to a linear system $\Phi$ of affine dimension $\rho_r$ of $(n-2)-$secant $\bbP^{n-3}$ to the rational normal curve $C_n$ in $\bbP^n$, then $\rank N^L_{n,k}=n-1-\rho_r$.
%
%The converse is generically true.
Then $\rank N^L_{n,k}=n-1-\rho^{n,k}_r$ if and only if the centre of projection $L\isom\bbP^{k-1}$ is contained in the base locus of a linear system $\Phi$ generated by $\rho^{n,k}_r$ linearly independent $\bbP^{n-3}$'s which intersect $C_n$ in degree $n-2$.  
\end{lem}
\begin{proof}
\begin{itemize}
\item[$\Leftarrow$] 
Let $L$ be a $\bbP^{k-1}$ which is contained in a linear system $\Phi=\{\pi_{\lambda}\isom\bbP^{n-3}:\pi_{\lambda}=\lambda_0\pi_0+...+\lambda_{\rho_r-1}\pi_{\rho^{n,k}_r-1},\;\;\forall\lambda=[\lambda_0,...,\lambda_{\rho^{n,k}_r-1}]\in\bbP^{\rho_r^{n,k}-1} \}$, where $\pi_0,...,\pi_{\rho^{n,k}_r-1}$ intersect $C_n$ in degree $n-2$.  Let $Q^i\subset C^0_n$ be the divisor such that $\pi_{i}=<Q_i>$. Then there exist $k$ points $p_1,...,p_{k}$ which generate $L$, such that each $p_i$ belongs to $\pi_{\lambda}$ for all $\lambda\in\bbP^{\rho^{n,k}_r-1}$. Let $Q_{\lambda}\subset C^0_n$ be the divisor such that $\pi_{\lambda}=<Q_{\lambda}>$, with $\mbox{Supp} Q_{\lambda}=<q^{m_1}_{1,\lambda},...,q^{m_s}_{s,\lambda}>$ and $s\leq n-2$, i.e. $Q_{\lambda}=m_1q_{1,\lambda}+...+m_sq_{s,\lambda}$. By Remark \ref{secantremark} the binary forms $f_i$ corresponding to $p_i$ can be decomposed as:
$$
f_i=c^i_{1,\lambda}L^n_{1,\lambda}+...+c^i_{n-2,\lambda}L^n_{n-2,\lambda},
$$
where $L_{j,\lambda}$ is the linear binary form corresponding in the usual way to $q_{j,\lambda}$ for $j=1,...,s$, so $L$ belongs to a $\bbP^{n-3\rho^{n,k}_r}$ (this is possible since  $\rho_r^{n,k}\leq (n-k+1)/3$). 
So by the Apolarity Lemma (see \cite{iarrobinokanev}) for each $\lambda\in\bbP^{\rho^{n,k}_r-1}$ there exists a differential form $\phi_{\lambda}\in T_{n-2}$ such that $\phi_{\lambda}\circ f_i=0$. Moreover there exist $\rho^{n,k}_r$ differential forms $\phi_0,...,\phi_{\rho^{n,k}_r-1}\in T_{n-2}$ and, for each $\lambda\in \bbP^{\rho^{n,k}_r-1}$, we have $\phi_{\lambda}=\lambda_0\phi_0+....+\lambda_{\rho^{n,k}_r-1}\phi_{\rho^{n,k}_r-1}$, in particular $\phi_j\circ f_i=0$ for all $j=0,...,\rho^{n,k}_r-1$ and $i=1,...,k$, so $\rank N^L_{n,k}=n-1-\rho^{n,k}_r$.

\item[$\Rightarrow$] If $\rank N^L_{n,k}=n-1-\rho^{n,k}_r$, then there exist $\rho^{n,k}_r$ binary form $\phi_0,...,\phi_{\rho^{n,k}_r-1}\in T_{n-2}$ such that however we consider the generating points $p_1,...,p_k\in \bbP^n$ of $L$, we have $\phi_{\alpha}\circ f_i=(\alpha_0\phi_0+...+\alpha_{\rho^{n,k}_r-1}\phi_{\rho^{n,k}_r-1})\circ f_i=0$ for all $\alpha=[\alpha_0,...,\alpha_{\rho^{n,k}_r-1}]\in\bbP^{\rho^{n,k}_r-1}$ and $i=1,...,k$, where $f_i\in S_n$ is the binary form corresponding to $p_i$. 
In particular $\phi_j\circ f_i=0$ for all $j=0,...,\rho^{n,k}_r-1$ and $i=1,...,k$, so we consider the primary decomposition of $\phi_{\alpha}=\prod^{n-2}_{l=1} \phi^l_{\alpha}$ and we indicate with $(L_{l,\alpha})^{\perp}=\phi^l_{\alpha}$. Therefore $f_1,...,f_k$ can be decomposed in $\infty^{\rho^{n,k}_r-1}$ different simultaneous ways, i.e.:
$$ 
f_i=c^i_{1,\alpha}L^n_{1,\alpha}+...+c^i_{n-2,\alpha}L^n_{n-2,\alpha},
$$
for all $\alpha=[\alpha_0,...,\alpha_{\rho^{n,k}_r-1}]\in \bbP^{\rho^{n,k}_r-1}$ or, in other words, $L$ is contained in the base locus of a linear system $\Phi=\{\pi_{\lambda}\isom\bbP^{n-2}:\pi_{\lambda}=\lambda_0\pi_0+...+\lambda_{\rho^{n,k}_r-1}\pi_{\rho^{n,k}_r-1},\;\;\forall\lambda=[\lambda_0,...,\lambda_{\rho^{n,k}_r-1}]\in\bbP^{\rho^{n,k}_r-1} \}$ generated by $\rho^{n,k}_r$ linearly independent $\bbP^{n-3}$'s which intersect $C_n$ in degree $n-2$. 
%appartiene al sistema lineare ossia all'intersezione di tutti gli spazi del sistema per cui sta su un $\bbP^{n-3\rho_r}$ se sono trasversi ovvero non è degenere 
So $L$ lies in a $\bbP^{n-3\rho^{n,k}_r}$ ( this is possible thanks to the condition $\rho^{n,k}_r\leq (n-k+1)/3$).
%Clearly it can happen that the binary differential forms $\phi_i$ have some multiple roots, so we have all possible degenerations of the linear system $\Phi$.
\end{itemize}
\end{proof}
By the above Lemma and Prop.\ref{prop:splittingnormal} we obtain:

\begin{teo}\label{mainthmn1}
Let $\rho^{n,k}_r\leq (n-k+1)/3$. The centre of projection $L\isom\bbP^{k-1}$ is contained in the base locus of a linear system $\Phi$ generated by $\rho^{n,k}_r$ linearly independent $\bbP^{n-3}$'s which intersect $C_n$ in degree $n-2$ if and only if $N_{\pi_L(C_n);\bbP^{n-k}}\isom \sO_{\pi_L(C_n)}(n+2)^{\rho^{n,k}_r}\oplus\mathcal{F}$, where:

\begin{itemize}
\item[i)] if $3k\geq n-1$ and $1\leq r \leq 2k-1 $, then $\mathcal{F}\isom\sO_{\pi_L(C_n)}(n+3)^{2k-2r}\oplus\mathcal{F}' $ with $\rank(\mathcal{F}')=r$ and $\deg({\mathcal{F}'}^{\vee}(n+2))=-2k$ ;

\item[ii)] if $3k<n-1$ and $1\leq r<n-k$, then $\mathcal{F}\isom\bigoplus^{\rank(N^L_{n,k})-k}_{i=0}\sO_{\pi_L(C_n)}(l_i)$ with $l_i\geq n+3$ and $\deg(\mathcal{F}^{\vee}(n+2))=-2k$.
\end{itemize}
\end{teo}

\subsection{Main Theorem for Normal Bundle}

By Theorem \ref{mainthmn1}, we can state the main result of this work:
\begin{teo}[Main Theorem for Normal Bundle]\label{mainthmn3}
Let $C=\pi_L(C_n)\subset \bbP^{n-k}$ and $\rho_r^{n,k}\leq (n-k+1)/3$. Then the following conditions are equivalent:
\begin{itemize}
\item[i)] we have 
$$
N_{\pi_{L}(C_n);\bbP^{n-k}}\isom \sO_{\bbP^1}(n+2)^{\rho^{n,k}_r}\oplus\mathcal{F},
$$ 
where $\deg(\mathcal{F}^{\vee}(n+2))=-2k$ and $\mathcal{F}\isom \bigoplus^{n-1-\rho^{n,k}_r-k}_i \sO_{\bbP^1}(l_i)$, with $l_i\geq n+3$;

\item[ii)] the centre of projection $L\isom\bbP^{k-1}$ is contained in the base locus of a linear system $\Phi$ generated by $\rho_r^{n,k}$ linearly independent $\bbP^{n-3}$'s which intersect $C_n$ in degree $n-2$.

\end{itemize}
Moreover if $L$ is general in the  variety of those $\bbP^{k-1}$ which are contained in the base locus of a linear system $\Phi$ generated by $\rho_r^{n,k}$ linearly independent $\bbP^{n-3}$'s which intersect $C_n$ in degree $n-2$, then 

$$
N_{\pi_L(C_n);\bbP^{n-k}}\isom \sO_{\bbP^1}(n+2)^{\rho^{n,k}_r}\oplus \sO_{\bbP^1}(n+2+B)^{A-2k+B\cdot A}\oplus\sO_{\bbP^1}(n+3+B)^{2k-B\cdot A}, 
$$
where $A=n-1-\rho^{n,k}_r-k$ and $B=\lfloor \frac{2k}{A}\rfloor$.
\end{teo}
\begin{proof}
The last part of the theorem follows by a result in \cite{ramanathan}, which allows us to calculate the generic splitting of a vector bundle $\mathcal{F}$ as in Theorem \ref{mainthmn1}:
$$
((n+2+B)^{A-2k+B\cdot A},(n+3+B)^{2k-B\cdot A}).
$$
where we have indicated $A=n-1-\rho^{n,k}_r-k$ and $B=\lfloor \frac{2k}{\rank(N^L_{n,k})-k}\rfloor$.
\end{proof}
\begin{propo}\label{codim.Normal}
Let $\rho^{n,k}_r\leq (n-k+1)/3$. % with $1< r\leq 2k-2$% 
The set of linear spaces $L\isom\bbP^{k-1}$ which lie in the base locus of a linear system $\Phi$ generated by $\rho^{n,k}_r$ linearly independent $\bbP^{n-3}$'s which intersect $C_n$ in $n-2$ distinct points, is an irreducible variety of codimension $\rho^{n,k}_r(3k-n+1+\rho^{n,k}_r)$ in $Gr(\bbP^{k-1},\bbP^n)$.

%:
%\begin{itemize}
%\item[i)] $N^n_{n-k}((n+2)^{\rho^{n,k}_r},(n+3)^{2k-2r},spt(\mathcal{F}' ))$ with $spt(\mathcal{F}' )$ is the splitting type, $\rank(\mathcal{F}')=r$ and $\deg({\mathcal{F}'}^{\vee}(n+2))=-2k$, for $3k\geq n-1$ and $1\leq r \leq 2k-1 $ ;
%
%\item[ii)] $N^n_{n-k}((n+2)^{r}, spt(\mathcal{F}))$, with $\mathcal{F}$ a vector bundle of rank $\rank(N^L_{n,k})-k$ on $\bbP^1$ and $\deg(\mathcal{F}^{\vee}(n+2))=-2k$ such that $\mathcal{F}\isom\bigoplus^{\rank(N^L_{n,k})-k}_{i=0}\sO(l_i)$ with $l_i\geq n+3$,   for $3k<n-1$ and $1\leq r<n-k$.
%\end{itemize}

\end{propo}
\begin{proof}
We can observe that the linear systems $\Phi$, of affine dimension $\rho^{n,k}_r$, made of $(n-2)-$secant $\bbP^{n-3}$'s to the rational normal curve $C_n$ in $\bbP^n$ correspond to the linear systems, of dimension $\rho^{n,k}_r$, made of binary forms of degree $n-2$; therefore the set of these linear systems corresponds to $Gr(\bbP^{\rho^{n,k}_r-1},\bbP^{n-2})$ which is irreducible with $\dim Gr(\bbP^{\rho^{n,k}_r-1},\bbP^{n-2})=\rho^{n,k}_r(n-1-\rho^{n,k}_r)$. We can compute the codimension of the variety of every $\bbP^{k-1}$ which belongs to some $(n-2)-$secant $\bbP^{n-3}$ via an incidence variety:
$$
I_S=\{(L,\pi):L\in Gr(\bbP^{k-1},\bbP^n),\pi\in S, L\subset S\},
$$ 
where $S$ is the set of all possible base loci of a linear system $\Phi$ generated by $\rho^{n,k}_r$ linearly independent $\bbP^{n-3}$'s which intersect $C_n$ in degree $n-2$. In the usual way we can compute the codimension of the image of this incidence variety in $Gr(\bbP^{k-1},\bbP^n)$. We will indicate with $\phi_1$ and $\phi_2$ the natural projections:
$$ 
\xymatrix{
& I_S \ar[dl]_{\phi_1} \ar[dr]^{\phi_2}&\\
Gr(\bbP^{k-1},\bbP^n) & & S,
}
$$
so the codimension in $Gr(\bbP^{k-1},\bbP^n)$ of $\phi_1(I_S)$ is equal to $\dim Gr(\bbP^{k-1},\bbP^n)-\dim S-\dim \phi^{-1}_2(S)$. 

Each projection linear space $L$ is contained in a $\bbP^{n-3\rho^{n,k}_r}$, so the dimension of the fibre is $k(n-3\rho^{n,k}_r+1-k)$, since $\rho^{n,k}_r\leq (n-k+1)/3$. Therefore the variety of linear spaces $L$ which lie in the base locus of a linear system $\Phi$ is an irreducible variety of codimension $\rho^{n,k}_r(3k-n+1+\rho^{n,k}_r)$ in $Gr(\bbP^{k-1},\bbP^n)$. Notice that the above calculation gives the actual dimension thanks to the result of Chiantini and Ciliberto \cite{chiantiniciliberto} on the non-defectivity of the Grassmannians of secant varieties of curves (see Theorem \ref{chiantiniciliberto} and Prop.\ref{dec.simultanee}).
\end{proof}
By Theorem \ref{mainthmn1} and the above Prop. \ref{codim.Normal} we can obtain the following result:
\begin{propo} 
Let $\rho^{n,k}_r\leq (n-k+1)/3$,$A=n-1-\rho^{n,k}_r-k$ and $B=\lfloor \frac{2k}{A}\rfloor$, then   

$N^n_{n-k}((n+2)^{\rho^{n,k}_r},(n+2+B)^{A-2k+B\cdot A},(n+3+B)^{2k-B\cdot A})$ is an open dense set of the irreducible variety of linear spaces $L\isom\bbP^{k-1}$ which lie in the base locus of a linear system $\Phi$ generated by $\rho^{n,k}_r$ linearly independent $\bbP^{n-3}$'s which intersect $C_n$ in $n-2$ distinct points in $Gr(\bbP^{k-1},\bbP^n)$.

\end{propo}

We want to point out that the summand $\sO_{\bbP^1}(n+2)$ is of particular interest because it has the smallest degree for ordinary singularity case. Thus we have proved the following theorem:
\begin{teo}\label{FinallyTheorem}
For any $0\leq \alpha \leq n-k-2$, the variety parametrizing rational curves $C$ of degree $n$ in $\bbP^{n-k}$ whose normal bundle $N_{C,\bbP^{n-k}}$ has the summand $\sO_C(n+2)^{\alpha}$ is irreducible.
\end{teo}

\subsection{Main Theorem for Restricted Tangent Bundle}

When considering the restricted tangent bundle on $C=\pi_L(C_n)\subset \bbP^{n-k}$,  we can notice that there are two different cases:
\begin{itemize}
\item[i)] $2k<n$, so $k<\frac{n}{2}$;
\item[ii)] $2k\geq n$, so $\frac{n}{2}\leq k<n-3$.
\end{itemize}
\begin{definition}
For all $n\in \bbN$, let us define $\delta^{n,k}_r \in\bbN$ as follows:
$$
\delta^{n,k}_r:=\left\lbrace
\begin{array}{cl}
r & \mbox{ , for } 2k\leq n \mbox{ and } 1\leq 2r\leq k-1;\\
n-3k+r-1 &\mbox{ , for } 2k > n \mbox{ and }  r \leq n-k-1.
\end{array}
\right.
$$

\end{definition}

\begin{remark}
In the above definition, the conditions on $r$ come from 
$$
\rank(T^L_{n,k})\geq k+1.
$$
\end{remark}

So we need the following proposition (see \cite{ramanathan}):
\begin{propo}
The generic splitting type of a vector bundle $\mathcal{F}$ on $\bbP^1$, with $\deg(\mathcal{F}^{\vee}(n+1))=-k$ and $\mathcal{F}\isom\bigoplus^{n-\delta^{n,k}_r-k}_{i=0}\sO_{\bbP^1}(l_i)$, with $l_i\geq n+2$, is the following:
$$
((n+1+B)^{A-k+B\cdot A},(n+2+B)^{k-B\cdot A}),
$$
where $A=n-\delta^{n,k}_r-k$ and $B=\lfloor \frac{k}{A}\rfloor$.
\end{propo}

In a similar way as for normal bundle, we obtain the following result which is stronger thanks to a result of Verdier (see \cite{verdier} and \cite{Ramella1}):
\begin{teo}[Main Theorem for Restricted Tangent Bundle]\label{mainthmtg3}

Let $\delta^{n,k}_r\leq (n-k+1)/2$. The following conditions are equivalent:
\begin{itemize}
\item[i)] we have 
$$
T^n_{n-k}((n+1)^{\delta^{n,k}_r},(n+1+B)^{A-k+B\cdot A},(n+2+B)^{k-B\cdot A}),
$$ 
where $A=n-\delta^{n,k}_r-k$ and $B=\lfloor \frac{k}{A}\rfloor$;
%$$
%T\bbP^{n-k}|_{\pi_L(C_n)}\isom \sO_{\bbP^1}(n+1)^{\delta_r}\oplus \mathcal{F}
%$$
%where $\deg(\mathcal{F}^{\vee}(n+1))=-k$ and $\mathcal{F}\isom\bigoplus^{n-\delta^{n,k}_r-k}_{i=0}\sO_{\bbP^1}(l_i)$, with $l_i\geq n+2$;

\item[ii)]  the centre of projection $L$ is contained in the base locus of a linear system $\Phi$ generated by $\delta^{n,k}_r$ linearly independent linearly independent $\bbP^{n-2}$'s which intersect $C_n\subset \bbP^n$ in $n-1$ distinct points.
\end{itemize}

\end{teo}
By a result in \cite{verdier} and \cite{Ramella1}:
\begin{cor}\label{maincortg3}
Let $\delta^{n,k}_r\leq (n-k+1)/2$.
 
$T^n_{n-k}((n+1)^{\delta^{n,k}_r},(n+1+B)^{A-k+B\cdot A},(n+2+B)^{k-B\cdot A})$, where $A=n-\delta^{n,k}_r-k$ and $B=\lfloor \frac{k}{A}\rfloor$,
is an irreducible variety of codimension $\delta^{n,k}_r(2k-n+\delta^{n,k}_r)$ in $Gr(\bbP^{k-1},\bbP^n)$, formed by the linear spaces $L\isom\bbP^{k-1}$ which lie in the base locus of a linear system $\Phi$ generated by $\delta^{n,k}_r$ linearly independent linearly independent $\bbP^{n-2}$'s which intersect $C_n\subset \bbP^n$ in $n-1$ distinct points.
\end{cor}

\section{Rational Space Curves of degree $5$}
In some particular cases we can compute the multi-degree of the varieties considered above thanks to the given geometric description.
In this section we will consider the case of rational space curves of degree $5$; this case has been fully investigated with regard to the restricted tangent bundle by Verdier (see \cite{verdier}) and Ramella (see \cite{Ramella1}).

By \cite{verdier} and our results in the previous section we have:
\begin{teo}
The centre of projection $L\isom\bbP^1$ is the common line of a pencil of $4-$secant $\bbP^{3}$'s to the rational normal curve $C_5$ in $\bbP^5$ if and only if :
$$
T\bbP^{3}|_{\pi_L(C_5)}\isom \sO_{\bbP^1}(6)^{2}\oplus\sO_{\bbP^1}(8).
$$
\end{teo}

\begin{cor}
The variety of lines $L$ that, as centre of projection, give a rational curve of degree $5$ in $\bbP^{3}$ which has $T\bbP^{3}_{\pi_L(C_5),\bbP^{3}}\isom\sO_{\bbP^1}(6)\oplus\sO_{\bbP^1}(6)\oplus\sO_{\bbP^1}(8)$ is an irreducible variety of codimension $2$ in $Gr(\bbP^1,\bbP^5)$, formed by the lines $L$ which are common base locus for a pencil of $4-$secant $\bbP^3$'s to the rational normal curve $C_5$ in $\bbP^5$.
\end{cor}

Now we want to analyze the splitting type of the normal bundle.
\begin{remark}
If $L\in Gr(\bbP^1,\bbP^5)$ does not lie on the Chow Hypersurface of the Secant Variety $\sigma_2(C_5)$ then $\pi_L(C_5)$ is smooth.
\end{remark}

We have that $deg(N^{\dual}_{\pi_L(C_5);\bbP^3}(7))=-4$ and $N^{\dual}_{\pi_L(C_5);\bbP^3}(7)\isom\sO_{\bbP^1} \oplus \sO_{\bbP^1}(-4)$ if and only if $N_{\pi_L(C_5);\bbP^3}(-5) \isom\sO_{\bbP^1}(2)\oplus \sO_{\bbP^1}(6)$. If we denote the following matrix as:

$$
N^L_{5,2}:=\left(
\begin{array}{cccccc}
a^1_0 & a^1_1 & a^1_2 & a^1_3 \\
-2a^1_1 & -2a^1_2 & -2a^1_3 & -2a^1_4 \\
a^1_2 & a^1_3 & a^1_4 & a^1_5 \\
a^2_0 & a^2_1 & a^2_2 & a^2_3 \\
-2a^2_1 & -2a^2_2 & -2a^2_3 & -2a^2_4 \\
a^2_2 & a^2_3 & a^2_4 & a^2_5 \\
\end{array}
\right),
$$

we have that $N_{\pi_L(C_5);\bbP^3}(-5)\isom\sO_{\bbP^1}(2) \oplus \sO_{\bbP^1}(6)$ if and only if $\rank(N^L_{5,2})= 3$  by Prop.\ref{prop:splittingnormal}.

\begin{propo}\label{propC5}
Let $C=\pi_L(C_5)\subset \bbP^3$ be a quintic space curves, with at most ordinary singularities. Then $N_{\pi_L(C_5);\bbP^3}\isom\sO_{\bbP^1}(7)\oplus \sO_{\bbP^1}(11)$ if and only if the centre of projection $L\in Gr(\bbP^1,\bbP^5)$ lies on a $3-$secant plane to $C_5$.  
\end{propo}
\begin{proof}
\begin{itemize}
\item[$\Leftarrow$]
If $L$ belongs to a $3-$secant plane, then there exist two points $p_1,p_2\in L$ such that the corresponding forms $f_1,f_2\in S_n$ have the same additive decomposition of length $3$:
$$
f_i=c^i_1(L_1)^5+c^i_2(L_2)^5+c^i_3(L_3)^5,
$$
for $i=1,2$ and $c^i_j\in \bbC$. Therefore we can construct a differential form $\phi=\prod^3_{j=1}(L_j)^{\perp}$ which generates $\ker N^L_{5,2}$, so $\rank N^L_{5,2}=3$ and by the above considerations $N_{\pi_L(C_5);\bbP^3}\isom\sO_{\bbP^1}(7)\oplus \sO_{\bbP^1}(11)$.
%\end{proof}
\item[$\Rightarrow$]
If the rational curve of degree $5$ projected from $L$ has $N_{\pi_L(C_5);\bbP^3}\isom\sO_{\bbP^1}(7)\oplus \sO_{\bbP^1}(11)$, then $\rank N^L_{5,2}=3$. Therefore there exists a differential form $\phi\in T_3$ apolar to both forms $f_1,f_2\in S_5$ corresponding to the points $p_1,p_2\in \bbP^5$ which generate $L$. It means that we have two possibilities: either $\phi$ has only simple roots or it has some multiple root.
So we must analyze three cases:
\begin{itemize}
\item[i)] $\phi$ has three simple roots; 
\item[ii)] $\phi$ has a simple root and a double root;
\item[iii)] $\phi$ has one triple root.
\end{itemize}

In the first case the primary decomposition is $\phi=\prod^3_{j=1}\phi_j$, so we have that the forms $f_1, f_2$ have similar additive decompositions:
$$
f_i=c^i_1(L_1)^5+c^i_2(L_2)^5+c^i_3(L_3)^5,
$$  
for $i=1,2$ and $c^i_j\in \bbC$. Here $L_j\in S_1$, $(L_j)^{\perp}=\phi_j$, $j=1,2,3$ and $L$ belongs to a $3-$secant plane generated by $q_1,q_2,q_3\in C_5$ corresponding to $(L_1)^5,(L_2)^5,(L_3)^5\in S_5$.

In the second case the primary decomposition is $\phi=\phi_1\phi^2_2$ with $\deg \phi_1=\deg \phi_2=1$, so we have:
$$
f_i=c^i_1(L_1)^5+G_{i,2}(L_2)^4,
$$
for $i=1,2$ and $c^i_1\in \bbC$. Here $L_j\in S_1$, $(L_j)^{\perp}=\phi_j$ for $j=1,2$ and $G_{i,2}\in S_1$. This means that $L$ belongs to the plane generated by $L^5_1$ and the line parametrized by $\{G\cdot L^4_2\in S_5 \mbox{ for all } G\in S_1 \}$ i.e the plane generated by a point on $C_5$ and a tangent line to $C_5$ in a different point. But in this case $L$ intersects a tangent line, so the projected curve would have a cusp, which is excluded by our preliminary assumption on the singularities.  

In the third case the primary decomposition is $\phi=\phi^3_1$ with $\deg \phi_1=1$; we have:
$$
f_i=G_i L^3,
$$
where $L\in S_1$, $L^{\perp}=\phi_1$ and $G_i\in S_2$. This means that $L$ belongs to the plane  parametrized by $\{G\cdot L^3\in S_5 \mbox{ for all } G\in S_2 \}$ i.e the osculating plane to $C_5$ in a point $q$ corresponding to $L^5\in S_5$. But in this case $L$ contains a tangent line, so the projected curve has a cusp, a case which is excluded by our preliminary assumption on the singularities.  

\end{itemize}
\end{proof}
\begin{lem}
The variety of lines $L\in Gr(\bbP^1,\bbP^5)$ which lie on a $3-$secant plane to the rational normal curve $C_5$ is an irreducible variety of codimension $3$.
\end{lem}
\begin{proof}
Infact we can consider the incidence variety $I_S=\{(L,\pi):L\in\bbGr(\bbP^1,\bbP^5),\pi\in S, L\subset S\}$ where $S$ is the set of all $3$-secant planes to the rational normal curve in $\bbP^5$. In the usual way we can compute the codimension of the image of this incidence variety in $Gr(\bbP^{1},\bbP^5)$. The above calculation is effective thanks to the result of Chiantini and Ciliberto on the non-defectivity of the Grassmannians of secant varieties of curves (see \cite{chiantiniciliberto}).
\end{proof}
Finally we show how to compute the multi-degree of this subvariety in the Grassmannian variety. 
\begin{lem}
The subvariety of $Gr(\bbP^1,\bbP^5)$ given by the lines which lie on a $3-$secant plane to $C_5$ has bidegree  $(1,6)$.
\end{lem}

\begin{proof}
If we fix a flag made of two non-empty linear subspaces of $\bbP^5$ : $\Lambda_0\subsetneq \Lambda_1$, the Schubert variety of this flag is the set:
$$
\Omega(\Lambda_0,\Lambda_1)=\{\Lambda\in Gr(\bbP^1,\bbP^5)\mid \dim (\Lambda\cap \Lambda_i)\geq i \mbox{ for } i=0,1\}.
$$
A class of equivalence of $\Omega(\Lambda_0,\Lambda_1)$ modulo projectivities is called a Schubert cycle and is denoted by $\Omega(l_0,l_1)$, where $l_i=\dim \Lambda_i$.

A special Schubert cycle is $\sigma_a=\Omega(4-a,5)$. The multidegree of the considered variety is the set of the degrees of the intersection numbers of its with all Schubert cycles of dimension 3, i.e. $\Omega(0,4)$ and $\Omega(1,3)$. With an easy computation we get that the first intersection number is $1$ and the second one is $6$.
\end{proof}
Hence we obtained the following:
\begin{teo}\label{thmC5}
$N^5_2(7,11)$ is an irreducible variety of codimension $3$ and bidegree $(1,6)$, formed by the lines $L\isom\bbP^{1}$ that lie on a $3-$secant plane to the rational normal curve $C_5$.
\end{teo}
%\begin{remark}
%If $L\in Gr(\bbP^1,\bbP^5)$ lies on a $3-$secant plane then it is contained in the Chow Hypersurface associated to the Secant Variety $\sigma_2(C_5)$.
%\end{remark}
\section*{Acknowledgements}
The results in this paper come from of my PhD thesis. I am very
grateful to my advisor Professor Giorgio Ottaviani for the patience with
which he followed this work very closely.

%\newpage
%\nocite{*}
\bibliography{BIBPHDTHESIS}
\bibliographystyle{plain}
%\bibhere

\end{document}